\pdfoutput=1 		
\documentclass[leqno]{article}
\usepackage[normal]{phaine_style}
\usepackage[margin=1.5in]{geometry}

\newcommand{\procategory}{\xspace{pro-$\infty$-cat\-e\-go\-ry}\xspace}
\newcommand{\procategories}{\xspace{pro-$\infty$-cat\-e\-gories}\xspace}

\newcommand{\Sigmaprimeclosed}{\xspace{$\Sigma'$-closed}\xspace}
\newcommand{\Sigmaclosed}{\xspace{$\Sigma$-closed}\xspace}
\newcommand{\Sigmacomplete}{\xspace{$\Sigma$-com\-plete}\xspace}
\newcommand{\Sigmacompletion}{\xspace{$\Sigma$-com\-ple\-tion}\xspace}
\newcommand{\Sigmacompletions}{\xspace{$\Sigma$-com\-ple\-tions}\xspace}
\newcommand{\Sigmagroup}{\xspace{$\Sigma$-group}\xspace}
\newcommand{\Sigmagroups}{\xspace{$\Sigma$-groups}\xspace}
\newcommand{\Sigmafinite}{\xspace{$\Sigma$-fi\-nite}\xspace}

\newcommand{\proSigma}{\xspace{pro-$\Sigma$}\xspace}

\newcommand{\RTop}{\categ{RTop}}

\newcommand{\BSigmacomp}{\Bup\Sigmacomp}
\newcommand{\Bprofin}{\Bup\profincomp}
\newcommand{\Btrun}{\Bup_{<\infty}}

\newcommand{\protrun}{\trun_{<\infty}}

\newcommand{\SpcSigma}{\Spc_{\Sigma}}
\newcommand{\ProSpcSigma}{\Pro(\SpcSigma)}
\newcommand{\Sigmacomp}{_{\Sigma}^{\wedge}}

\newcommand{\cospecializes}{\leftsquigarrow}

\newcommand{\profincomp}{_{\uppi}^{\wedge}}

\newcommand{\pietprofin}{\hat{\uppi}^{\et}}

\DeclareMathOperator{\BGal}{BGal}
\newcommand{\Gk}{\Gup_{k}}
\newcommand{\BG}{\mathrm{BG}}
\newcommand{\BGk}{\BG_{k}}

\newcommand{\Piet}{\Pi_{\infty}^{\et}}
\newcommand{\Pietprotrun}{\Pi_{<\infty}^{\et}}
\newcommand{\Pietprofin}{\widehat{\Pi}_{\infty}^{\et}}

\DeclareMathOperator{\qcqs}{qcqs}
\newcommand{\Schqcqs}{\Sch^{\qcqs}}

\newcommand{\modmod}{\mathbin{\sslash}}

\newcommand{\ProCat}{\Pro(\Catinfty)}

\newcommand{\ProSpc}{\Pro(\Spc)}
\newcommand{\Spcfin}{\Space_{\uppi}}
\newcommand{\Spctrun}{\Space_{<\infty}}
\newcommand{\ProSpcfin}{\Pro(\Spcfin)}
\newcommand{\ProSpctrun}{\Pro(\Spctrun)}

\DeclareMathOperator{\Nbd}{Nbd}

\newcommand{\Xkbar}{X_{\kbar}}
\newcommand{\Xsbar}{X_{\sbar}}
\newcommand{\Xsbaret}{X_{\sbar,\et}}
\newcommand{\Loc}[2]{{#1}_{(#2)}}
\newcommand{\Locet}[2]{{#1}_{(#2),\et}}

\title{\Large The fundamental fiber sequence in étale homotopy theory}

\author{\normalsize Peter J. Haine \and\normalsize Tim Holzschuh \and\normalsize Sebastian Wolf}

\date{\normalsize \today}

\begin{document}

\maketitle


\begin{abstract} 
	Let $ k $ be a field with separable closure $ \kbar \supset k $, and let $ X $ be a qcqs $ k $-scheme.
	We use the theory of profinite Galois categories developed by Barwick--Glasman--Haine to provide a quick conceptual proof that the sequences
	\begin{equation*}
		\Pietprotrun(\Xkbar) \to \Pietprotrun(X) \to \BGal(\kbar/k) \andeq \Pietprofin(\Xkbar) \to \Pietprofin(X) \to \BGal(\kbar/k)
	\end{equation*}
	of protruncated and profinite étale homotopy types are fiber sequences.
	This gives a common conceptual reason for the following two phenomena:
	first, the higher étale homotopy groups of $ X $ and the geometric fiber $ \Xkbar $ are isomorphic, and
	second, if $ \Xkbar $ is connected, then the sequence of profinite étale fundamental groups \smash{$ 1 \to \pietprofin_1(\Xkbar) \to \pietprofin_1(X) \to \Gal(\kbar/k) \to 1 $} is exact.
	It also proves the analogous results for the \textit{groupe fondamental élargi} of SGA3.
\end{abstract}



\setcounter{tocdepth}{1}

\tableofcontents


\setcounter{section}{-1}


\section{Introduction}

Let $ k $ be a field with separable closure $ \kbar \supset k $, and let $ X $ be a qcqs $ k $-scheme.
Write $ \Xkbar $ for the basechange of $ X $ to $ \kbar $.
The \textit{fundamental exact sequence} for étale fundamental groups asserts that if $ \Xkbar $ is connected, then the natural sequence of profinite groups
\begin{equation}\label{seq:fundamental_exact_sequence}
	\begin{tikzcd}[sep=1.5em]
		1 \arrow[r] & \pietprofin_1(\Xkbar) \arrow[r] & \pietprofin_1(X) \arrow[r] & \Gal(\kbar/k) \arrow[r] & 1
	\end{tikzcd}
\end{equation}
is exact \cites[\stackstag{0BTX}]{stacksproject}[Exposé IX, Théorème 6.1]{MR50:7129}.
The purpose of this paper is to explain a simple argument that extends this short exact sequence to a fiber sequence of étale homotopy types in the sense of Artin--Mazur--Friedlander \cites{MR0245577}[Chapter 4]{MR676809}.
To do this, we make use of Barwick, Glasman, and Haine's new description of the étale homotopy type in terms of profinite Galois categories \cites[Theorem 12.5.1]{arXiv:1807.03281}{arXiv:1812.11637}.

Given a scheme $ Y $, there are two variants of the étale homotopy type relevant to the present work: the first is the (protruncated) étale homotopy type \smash{$ \Pietprotrun(Y) $}.
This is a prospace which is not generally profinitely complete.
The fundamental progroup of \smash{$ \Pietprotrun(Y) $} recovers the \textit{groupe fondamental élargi} of SGA3 \cite[Exposé X, \S6]{MR43:223b}.
The second is the profinite completion of \smash{$ \Pietprotrun(Y) $}, which we denote by \smash{$ \Pietprofin(Y) $}.
The fundamental progroup of \smash{$ \Pietprofin(Y) $} recovers the usual profinite étale fundamental group \smash{$ \pietprofin_1(Y) $}.
The profinite étale homotopy type also comes equipped with a natural comparison morphism \smash{$ \fromto{\Pietprotrun(Y)}{\Pietprofin(Y)} $}.
If $Y$ is geometrically unibranch, both variants coincide: in this case, the comparison morphism $ \fromto{\Pietprotrun(Y)}{\Pietprofin(Y)} $ is an equivalence \cites[Theorem 3.6.5]{DAGXIII}[Theorem 11.1]{MR0245577}[Theorem 7.3]{MR676809}.
See \cref{subsec:etale_homotopy_types_Galois_categories} for a short recollection of the modern perspective on the étale homotopy type and how it relates to the classical definition.

The following are the main results of this paper.
We note that with our methods, the only thing particular to working over a field that we need is that $ \Spec(k) $ is $ 0 $-dimensional.
Write $ \ProSpc $ for the \category of prospaces and \smash{$ \ProSpcfin \subset \ProSpc $} for the full subcategory spanned by the profinite spaces.

\begin{theorem}[(\Cref{cor:protruncated_fundamental_fiber_sequence})]\label{intro_thm:protruncated_fundamental_fiber_sequence}
	Let $ f \colon \fromto{X}{S} $ be a morphism between qcqs schemes, and let $ \fromto{\sbar}{S} $ be a geometric point of $ S $.
	If $ \dim(S) = 0 $, then the naturally null sequence
	\begin{equation}\label{intro_eq:fundamental_fiber_sequence_protruncated}
		\begin{tikzcd}[sep=1.5em]
			\Pietprotrun(\Xsbar) \arrow[r] & \Pietprotrun(X) \arrow[r] & \Pietprotrun(S) 
		\end{tikzcd}
	\end{equation}
	is a fiber sequence in the \category $ \ProSpc $.
\end{theorem}

\noindent For a $ 0 $-dimensional scheme $ S $, the prospace $ \Pietprotrun(S) $ is already profinitely complete (see \Cref{ex:Piet_of_0-dimensional_is_1-truncated}).
Hence the following property of profinite completion implies that \eqref{intro_eq:fundamental_fiber_sequence_protruncated} remains a fiber sequence after profinite completion.

\begin{proposition}[(\Cref{prop:Sigma-completion_locally_cartesian})]\label{intro_prop:profinite_completion_locally_cartesian}
	The functor $(-)\profincomp \colon \fromto{\ProSpc}{\ProSpcfin} $ that carries a prospace to its profinite completion preserves pullbacks along maps between profinite spaces.
\end{proposition}

\begin{corollary}[(\Cref{cor:profinite_fundamental_fiber_sequence})]\label{intro_cor:profinite_fundamental_fiber_sequence}
	With the same notation as \Cref{intro_thm:protruncated_fundamental_fiber_sequence}, if $ \dim(S) = 0 $, then the naturally null sequence
	\begin{equation*}\label{intro_eq:fundamental_fiber_sequence_profinite}
		\begin{tikzcd}[sep=1.5em]
			\Pietprofin(\Xsbar) \arrow[r] & \Pietprofin(X) \arrow[r] & \Pietprofin(S) 
		\end{tikzcd}
	\end{equation*}
	is a fiber sequence in the \category $ \ProSpcfin$.
\end{corollary}

By taking homotopy groups, \Cref{intro_cor:profinite_fundamental_fiber_sequence} recovers the fundamental exact sequence \eqref{seq:fundamental_exact_sequence}.
To explain this, we first introduce the following notation:

\begin{notation}
	Given a qcqs scheme $ Y $ and geometric point $ \fromto{\ybar}{Y} $, write $ \uppi_0(Y) $ for the profinite set of connected components of $ \Pietprofin(Y) $ and $ \pietprofin_n(Y,\ybar) $ for the $ n $-th profinite homotopy group of $ \Pietprofin(Y) $ at $ \ybar $.
	(Equivalently, $ \uppi_0(Y) $ is the profinite set of connected components of $ Y $.)
\end{notation}

\noindent Choose a geometric point $ \xbar $ of $ \Xsbar $ with image $ x $ in $ X $.
Since the higher étale homotopy groups of a $ 0 $-dimensional qcqs scheme vanish (see \Cref{ex:Piet_of_0-dimensional_is_1-truncated}), \Cref{intro_cor:profinite_fundamental_fiber_sequence} shows that the higher étale homotopy groups of $ X $ are \textit{geometric}: for $ n \geq 2 $, the natural homomorphism of profinite groups
\begin{equation*}
	\fromto{\pietprofin_n(\Xsbar,\xbar)}{\pietprofin_n(X,x)}
\end{equation*}
is an isomorphism.
Moreover, without assuming that the geometric fiber $ \Xsbar $ is connected, we obtain an exact sequence of pointed profinite sets
\begin{equation*}\label{seq:general_fundamental_exact_sequence_with_pi0}
	\begin{tikzcd}[sep=1.5em]
		1 \arrow[r] & \pietprofin_1(\Xsbar,\xbar) \arrow[r] & \pietprofin_1(X,x) \arrow[r] & \pietprofin_1(S,\sbar) \arrow[r] & \uppi_0(\Xsbar) \arrow[r] & \uppi_0(X) \arrow[r] & \uppi_0(S) \period
	\end{tikzcd}
\end{equation*}
\Cref{intro_thm:protruncated_fundamental_fiber_sequence} implies the analogous isomorphisms and exact sequence in the category of pointed prosets for the homotopy prosets of the protruncated étale homotopy types of $ \Xsbar $, $ X $, and $ S $.
See also \cite[Corollary 4.10]{arXiv:1910.14015}.

\begin{example}
	Take $ S $ to be the spectrum of a field $ k $ with separable closure $ \kbar \supset k $.
	\Cref{intro_cor:profinite_fundamental_fiber_sequence} provides a natural fiber sequence of profinite étale homotopy types
	\begin{equation*}
		\begin{tikzcd}[sep=1.5em]
			\Pietprofin(\Xkbar) \arrow[r] & \Pietprofin(X) \arrow[r] & \BGal(\kbar/k) \period
		\end{tikzcd}
	\end{equation*}
	As a consequence, the sequence of pointed profinite sets
	\begin{equation*}\label{seq:fundamental_exact_sequence_with_pi0}
		\begin{tikzcd}[sep=1.5em]
			1 \arrow[r] & \pietprofin_1(\Xkbar,\xbar) \arrow[r] & \pietprofin_1(X,x) \arrow[r] & \Gal(\kbar/k) \arrow[r] & \uppi_0(\Xkbar) \arrow[r] & \uppi_0(X) \arrow[r] & 1
		\end{tikzcd}
	\end{equation*}
	is exact.
	Since we do not make use of the fundamental exact sequence \eqref{seq:fundamental_exact_sequence}, \Cref{intro_cor:profinite_fundamental_fiber_sequence} provides a new proof of the fundamental exact sequence.
\end{example}


\subsection{Proof overview}\label{subsec:proof_overview}

To prove \Cref{intro_thm:protruncated_fundamental_fiber_sequence}, we use Barwick, Glasman, and Haine's description of the étale homotopy type in terms of profinite Galois categories \cites[Theorem 12.5.1]{arXiv:1807.03281}{arXiv:1812.11637}.
Let us briefly recall this description.
Given a qcqs scheme $ Y $, Barwick--Glasman--Haine gave the category of points of the étale topos of $ Y $ the structure of a \proobject in the category of categories with finitely many morphisms.
Since it globalizes the absolute Galois groups of the residue fields of the points of $ Y $, they denote the resulting procategory by $ \Gal(Y) $.
Using Hoyois' description of the étale homotopy type \cite[Corollary 5.6]{MR3763287} via Lurie's shape theory \cites[\HTTsubsec{7.1.6}]{HTT}[\HAsec{A.1}]{HA}[\SAGsec{E.2}]{SAG}, Barwick--Glasman--Haine showed that the prospace $ \Pietprotrun(Y) $ can be recovered as the protruncated classifying space of the procategory $ \Gal(Y) $.
See \cref{subsec:etale_homotopy_types_Galois_categories} for more details.

Via this perspective, proving \Cref{intro_thm:protruncated_fundamental_fiber_sequence} amounts to showing that a sequence of classifying prospaces is a fiber sequence.
The geometric input we need is the following:
for any morphism between qcqs schemes $ f \colon \fromto{X}{S} $ and geometric point $ \fromto{\sbar}{S} $, the sequence of profinite categories
\begin{equation*}
	\begin{tikzcd}[sep=1.5em]
		\Gal(\Xsbar) \arrow[r] & \Gal(X) \arrow[r] & \Gal(S)
	\end{tikzcd}
\end{equation*}
is a fiber sequence (see \cref{sec:Galois_categories_of_fibers}).
If $ \dim(S) = 0 $, then the profinite category $ \Gal(S) $ is already a profinite $ 1 $-groupoid.
\Cref{intro_thm:protruncated_fundamental_fiber_sequence} then follows from the assertion that taking protruncated classifying spaces preserves pullbacks along morphisms between profinite $ 1 $-groupoids. 
In \cref{sec:fiber_sequence}, we prove these categorical facts as well as \Cref{intro_prop:profinite_completion_locally_cartesian}.
See \Cref{ex:classifying_prospace_locally_cartesian,cor:protrun_classifying_space_is_localy_cartesian}.


\subsection{Related work}\label{subsec:related_work}

Let $ k $ be a field with separable closure $ \kbar \supset k $, and let $ X $ be a qcqs $ k $-scheme.
Write $ \Gk \colonequals \Gal(\kbar/k) $.
\Cref{intro_thm:protruncated_fundamental_fiber_sequence} generalizes work of Schmidt--Stix.
In the proof of \cite[Proposition 2.3]{MR3549624}, Schmidt and Stix showed that the sequence of protruncated étale homotopy types
\begin{equation*}
	\begin{tikzcd}[sep=1.5em]
		\Pietprotrun(\Xkbar) \arrow[r] & \Pietprotrun(X) \arrow[r] & \BGk
	\end{tikzcd}
\end{equation*}
is a fiber sequence, provided that $ X $ is separated, locally noetherian, and of finite type over $ k $.
Their proof uses Friedlander's description of the étale homotopy type of a locally noetherian scheme via rigid hypercovers.
It also strongly relies on the assumptions that $ X $ is of finite type and that the base is a field. 
At the time it was not known if their work implied \Cref{intro_cor:profinite_fundamental_fiber_sequence} (under these assumptions); \Cref{intro_prop:profinite_completion_locally_cartesian} shows that this is indeed the case.

\Cref{intro_cor:profinite_fundamental_fiber_sequence} generalizes work of Cox, Quick, and Chough.
Extending work of Cox over $ \RR $ \cite[Theorem 1.1]{MR534381}, and Quick for \textit{varieties} over general fields \cite[Theorem 3.5]{MR2747236}, Chough showed the natural map
\begin{equation*}
	\fromto{\Pietprofin(\Xkbar)}{\Pietprofin(X)}
\end{equation*}
realizes $ \Pietprofin(X) $ as the quotient $ \Pietprofin(\Xkbar) \modmod \Gk $ of the profinite étale homotopy type of $ \Xkbar $ by the natural $ \Gk $-action \cite[Theorem 5.1.26]{MR3553672}.
Chough's proof uses the relative étale homotopy type \cites[\S9.2.3]{MR3077173}[\S8.1]{MR3459031}.

Since Chough's thesis, Lurie proved the following: given a profinite group $ G $, there is an equivalence of \categories between profinite spaces with a continuous $ G $-action and profinite spaces with a map to the profinite classifying space $ \Bup G $ \SAG{Theorem}{E.6.5.1}.
This equivalence sends a profinite space $ U $ with $ G $-action to the quotient $ U \modmod G $ and a map of profinite spaces $ \phi \colon \fromto{V}{\Bup G} $ to the fiber $ \fib(\phi) $ over the unique point of $ \Bup G $.
In light of this dictionary, \Cref{intro_cor:profinite_fundamental_fiber_sequence} is equivalent to the presentation $ \Pietprofin(X) \equivalent \Pietprofin(\Xkbar) \modmod \Gk $.
Note that our method of proof is completely different from Chough's and works over more general bases.



\begin{acknowledgments}
	We thank Clark Barwick, Elden Elmanto, and Marc Hoyois for many enlightening discussions around the contents of this paper.
	We also thank Alexander Schmidt for comments on a draft of this work.

	We would like to thank the SFB 1085 `Higher Invariants' and the University of Regensburg for its hospitality.
	The first-named author gratefully acknowledges support from the UC President's Postdoctoral Fellowship, NSF Mathematical Sciences Postdoctoral Research Fellowship under Grant \#DMS-2102957, and a grant from the Simons Foundation (816048, LC). 
	The second-named author acknowledges support from the Deutsche Forschungsgemeinschaft (DFG) through the Collaborative Research Centre TRR 326 `Geometry and Arithmetic of Uniformized Structures',  project number 444845124.
	The third-named author was supported by the SFB 1085 `Higher Invariants' in Regensburg, funded by the DFG.
\end{acknowledgments}



\section{Background}\label{sec:etale_homotopy_types_Galois_categories}

We begin by collecting some background and notation on \proobjects, étale homotopy types, and profinite Galois categories.


\subsection{\Proobjects}\label{subsec:background_pro-objects}

In this subsection, we set our notation for \proobjects and the various completion functors relating the \categories of \proobjects relevant to this paper.
We refer the unfamiliar reader to \cite[\SAGsubsec{A.8.1}]{SAG} for more background on \proobjects, \cites[\S4.1]{arXiv:1807.03281}[\S3]{arXiv:1810.00351} for background on protruncated objects, and \cites[\SAGapp{E}]{SAG}[\S4.4]{arXiv:1807.03281} for background on profinite spaces.

\begin{notation}
	We write $ \Spc $ for the \category of spaces and $ \Catinfty $ for the \category of \categories.
\end{notation}

\begin{notation}
	Given \acategory $ \Ccal $, we write $ \Pro(\Ccal) $ for the \category of \proobjects in $ \Ccal $ obtained by formally adjoining cofiltered limits to $ \Ccal $.
	The existence of $ \Pro(\Ccal) $ is a special case of (the dual of) \HTT{Proposition}{5.3.6.2}.
	Given a functor $ F \colon \fromto{\Ccal}{\Dcal} $, we simply write $ F \colon \fromto{\Pro(\Ccal)}{\Pro(\Dcal)} $ for the cofiltered-limit-preserving extension of $ F $.
\end{notation}

\begin{nul}
	Note that an adjunction $ \adjto{L}{\Ccal}{\Dcal}{R} $ extends along cofiltered limits to an adjunction $ \adjto{L}{\Pro(\Ccal)}{\Pro(\Dcal)}{R} $.
\end{nul}

\begin{observation}\label{obs:existence_of_materalization}
	If $ \Ccal $ admits cofiltered limits, then the identity $ \fromto{\Ccal}{\Ccal} $ extends to a cofiltered-limit-preserving functor $ \lim \colon \fromto{\Pro(\Ccal)}{\Ccal} $.
	This functor sends a prosystem $ \{U_{i}\}_{i \in I} $ to the limit $ \lim_{i \in I} U_i $ computed in $ \Ccal $.
	Moreover, the functor $ \lim \colon \fromto{\Pro(\Ccal)}{\Ccal} $ is right adjoint to the Yoneda embedding $ \incto{\Ccal}{\Pro(\Ccal)} $.
\end{observation}

We are mostly interested in (localizations of) the \categories $ \ProCat $ of \procategories and $ \ProSpc $ of prospaces.
Equivalences in $ \ProSpc $ cannot be detected on homotopy prosets; thus one wants to work with the localization of $ \ProSpc $ at the $ \uppi_{\ast} $-isomorphisms.
Since there are nontrivial prospaces with no points, instead of working with homotopy progroups, it is better to work with truncations.

\begin{notation}
	Given an integer $ n \geq 0 $, write $ \Spc_{\leq n} \subset \Spc $ for the full subcategory spanned by the $ n $-truncated spaces.
	Write $ \trun_{\leq n} \colon \fromto{\Spc}{\Spc_{\leq n}} $ for the left adjoint to the inclusion.
	Given a space $ U $ we call $ \trun_{\leq n}(U) $ the \defn{$ n $-truncation} of $ U $.

	We say that a space $ U $ is \defn{truncated} if $ U $ is $ n $-truncated for some integer $ n \geq 0 $.
	We write $ \Spctrun \subset \Spc $ for the full subcategory spanned by the truncated spaces.
\end{notation}

\begin{notation}[(protruncation)]
	The inclusion $ \ProSpctrun \subset \ProSpc $ admits a left adjoint
	\begin{equation*}
		\protrun \colon \fromto{\ProSpc}{\ProSpctrun}
	\end{equation*}
	defined as follows.
	The functor $ \protrun $ is the unique cofiltered-limit-preserving extension of the fully faithful functor $ \incto{\Spc}{\ProSpctrun} $ that sends a space $ U $ to the cofiltered diagram given by its Postnikov tower $ \{\trun_{\leq n}(U)\}_{n\geq 0} $.
	We refer to $ \ProSpctrun $ as the \category of \defn{protruncated spaces} and $ \protrun $ as the \defn{protruncation} functor.
\end{notation}

\begin{nul}\label{nul:equivalences_on_protruncations}
	Said differently, a map of prospaces $ \fromto{U}{V} $ becomes an equivalence after protruncation if and only if for each $ n \geq 0 $, the induced map of prospaces $ \fromto{\trun_{\leq n}(U)}{\trun_{\leq n}(V)} $ is an equivalence. 
\end{nul}

\begin{nul}
	By \cites[Remark 3.2]{arXiv:1810.00351}[Corollary 7.5]{MR1828474}, a map of pointed \textit{connected} prospaces $ \fromto{U}{V} $ becomes an equivalence after protruncation if and only if for each $ n \geq 1 $, the induced map of homotopy progroups $ \fromto{\uppi_n(U)}{\uppi_n(V)} $ is an isomorphism.
\end{nul}

We are also interested in \textit{profinite completions} of prospaces.

\begin{notation}[(profinite completion)]
	A space $ U $ is \defn{\pifinite} if $ U $ is truncated, $ \uppi_0(U) $ is finite, and all homotopy groups of $ U $ are finite.
	We write $ \Spcfin \subset \Spc $ for the full subcategory spanned by the \pifinite spaces.
	Again, the inclusion $ \ProSpcfin \subset \ProSpc $ admits a left adjoint
	\begin{equation*}
		(-)\profincomp \colon \fromto{\ProSpc}{\ProSpcfin} \period
	\end{equation*}
	See \SAG{Remark}{E.2.1.3}.
	We call $ \ProSpcfin $ the \category of \defn{profinite spaces} and $ (-)\profincomp $ the \defn{profinite completion} functor.
	Note that since $ \Spcfin \subset \Spctrun $, the profinite completion functor factors through $ \ProSpctrun $.
\end{notation}

We are also interested in various types of \textit{classifying spaces} for \procategories.

\begin{notation}
	We denote the left adjoint to the inclusion $ \Spc \subset \Catinfty $ by $ \Bup \colon \fromto{\Catinfty}{\Spc} $.
	Given \acategory $ \Ccal $, we call $ \Bup\Ccal $ the \defn{classifying space} of $ \Ccal $.
\end{notation}

\noindent We make use of the description of classifying spaces as geometric realizations.

\begin{recollection}\label{rec:complete_Segal_spaces}
	The nerve construction defines a fully faithful right adjoint
	\begin{equation*}
		\incto{\Catinfty}{\Fun(\Deltaop,\Spc)}
	\end{equation*}
	from the \category of \categories to the \category of simplicial spaces \cites[\HAappthm{Proposition}{A.7.10}]{HA}[\SAGsubsec{A.8.2}]{SAG}{MR2342834}[\S1]{Lurie:Goodwillie-I}{MR1804411}.
	Objects in the image of this embedding are often called \textit{complete Segal spaces}.
	Under this embedding, the subcategory $ \Spc \subset \Catinfty $ corresponds to the constant functors $ \fromto{\Deltaop}{\Spc} $.
	Moreover, the localization $ \Bup \colon \fromto{\Catinfty}{\Spc} $ is given by geometric realization.
\end{recollection}

\begin{notation}[(classifying prospaces)]
	Write $ \Btrun $ for the composite
	\begin{equation*}
		\begin{tikzcd}[sep=2em]
			\ProCat \arrow[r, "\Bup"] & \ProSpc \arrow[r, "\protrun"] & \ProSpctrun \period
		\end{tikzcd}
	\end{equation*}
	The functor $ \Btrun $ is left adjoint to the inclusion $ \ProSpctrun \subset \ProCat $.
	Given a \procategory $ \Ccal $, we refer to \smash{$ \Btrun(\Ccal) $} as the \defn{protruncated classifying space} of $ \Ccal $.
	Write \smash{$ \Bprofin $} for the composite
	\begin{equation*}
		\begin{tikzcd}[sep=2em]
			\ProCat \arrow[r, "\Bup"] & \ProSpc \arrow[r, "(-)\profincomp"] & \ProSpcfin \period
		\end{tikzcd}
	\end{equation*}
	The functor $ \Bprofin $ is left adjoint to the inclusion $ \ProSpcfin \subset \ProCat $.
	Given a \procategory $ \Ccal $, we refer to \smash{$ \Bprofin(\Ccal) $} as the \defn{profinite classifying space} of $ \Ccal $.
\end{notation}


\subsection{Étale homotopy types \& Galois categories}\label{subsec:etale_homotopy_types_Galois_categories}

We now set our conventions for étale homotopy types and their refinements to profinite Galois categories.
For background on étale homotopy types, the unfamiliar reader should refer to \cites{MR0245577}{MR3459031}[Chapter 4]{MR676809}{MR3077173}{MR3077166} for the more classical perspective and to \cites[Chapters 4 \& 11]{arXiv:1807.03281}[\S2]{arXiv:1905.06243}[\S2]{MR4367219}{MR3763287}{arXiv:1810.00351} for the more modern perspective using Lurie's shape theory.
The reader should refer to \cite[Chapter 12]{arXiv:1807.03281} for more background on profinite Galois categories.

We begin by recalling a bit about the modern interpretation of the étale homotopy type.
The point is that the original definition only made sense for locally noetherian schemes, but Lurie's shape theory allows one to define the étale homotopy type of arbitrary schemes.
We emphasize that the reader does not need to be familiar with \topoi or shape theory to understand the proofs in this paper; all of our results make use of the description of the étale homotopy type provided by \Cref{thm:exodromy_definition_of_Piet}. 

\begin{recollection}
	Let $ Y $ be a locally noetherian scheme.
	Using hypercovers, Artin and Mazur \cite[\S 9]{MR0245577} constructed a \proobject in the homotopy category of spaces called the \textit{étale homotopy type} of $ Y $.
	Friedlander \cite[\S 4]{MR676809} refined this construction, producing a \proobject in simplicial sets which he called the \textit{étale topological type} of $ Y $.
	Hoyois provided a modern interpretation of Friedlander's construction: Friedlander's étale topological type corepresentes the shape of the \topos of étale \textit{hyper}sheaves of spaces on $ Y $ \cite[Corollary 5.6]{MR3763287}.
\end{recollection}

\begin{remark}
	From the modern perspective, it is more natural to consider the shape of the \topos of étale sheaves of spaces (with no hyperdescent conditions) on $ Y $.
	This is only a minor departure from the Artin--Mazur--Friedlander étale homotopy type: by \cite[Example 4.2.8]{arXiv:1807.03281}, the protruncations of the shapes of the \topoi of étale hypersheaves and étale sheaves on $ Y $ agree.
\end{remark}

Since Lurie's shape theory makes sense for arbitrary \topoi, it provides a definition of the étale homotopy type of any scheme.

\begin{notation}
	Given a scheme $ Y $, we write $ \Piet(Y) \in \ProSpc $ for the shape of the \topos of étale sheaves of spaces on $ Y $.
	We simply refer to $ \Piet(Y) $ as the \defn{étale homotopy type} of $ Y $.
	We write $ \Pietprotrun(Y) \in \ProSpctrun $ for the protruncation of $ \Piet(Y) $, and write \smash{$ \Pietprofin(Y) $} for the profinite completion of \smash{$ \Piet(Y) $}.
\end{notation}

Now we set the context for profinite Galois categories.
To do this, we need to fix some notation and recall a bit about points in the étale topology.

\begin{notation}
	We write $ \RTop $ for the $ (2,1) $-category of topoi and (right adjoints in) geometric morphisms. 
	For a scheme $ Y $, we write $ Y_{\et} $ for the small étale topos of $ Y $.
	Given a morphism of schemes $ f \colon X \to S $, we write $ \flowerstar \colon X_{\et} \to S_{\et} $ for the induced geometric morphism of étale topoi.
\end{notation}

\begin{notation}\label{ntn:geometric_points}
	Let $ Y $ be a scheme and $ \fromto{\ybar}{Y} $ a geometric point.
	We write $ y \in Y $ for the underlying point of $ \ybar $.
\end{notation}

\begin{recollection}\label{rec:points_of_etale_topos}
	Let $ Y $ be a qcqs scheme. 
	The Grothendieck School \cite[Exposé VIII, Théorème 7.9]{MR50:7131} computed the category $ \Pt(Y_{\et}) $ of points of the étale topos of $ Y $:
	\begin{enumerate}[label=\stlabel{rec:points_of_etale_topos}, ref=\arabic*]
		\item Objects of $ \Pt(Y_{\et}) $ are geometric points $ \fromto{\ybar}{Y} $.

		\item Given geometric points $ \fromto{\sbar}{Y} $ and $ \fromto{\etabar}{Y} $ a morphism $ \fromto{\sbar}{\etabar} $ in $ \Pt(Y_{\et}) $ is an \textit{étale specialization} $ \sbar \cospecializes \etabar $: a morphism of $ Y $-schemes \smash{$ \fromto{\Spec(\Ocal_{Y,\eta}^{\sh})}{\Spec(\Ocal_{Y,s}^{\sh})} $} between spectra of strictly henselian local rings.
	\end{enumerate}
	Importantly, there is a natural isomorphism of \emph{sets}
	\begin{equation}\label{intro_eq:endomorphisms_Galois_group}
		\Hom_{\Pt(Y_{\et})}(\ybar,\ybar) \isomorphic \Gal(\upkappa(\ybar)/\upkappa(y)) \period
	\end{equation}
\end{recollection}

Barwick--Glasman--Haine gave $ \Pt(Y_{\et}) $ the structure of a \textit{profinite} category:

\begin{notation}
	We say that a $ 1 $-category $ \Ccal $ is \defn{finite} if $ \Ccal $ has finitely many objects up to isomorphism and finite $ \Hom $ sets. 
	We write $ \Cat_{1,\uppi} \subset \Catinfty $ for the full subcategory spanned by the finite $ 1 $-categories and refer to objects of $ \Pro(\Cat_{1,\uppi}) $ as \defn{profinite categories}.
\end{notation}

\begin{remark}
	Since the inclusion $ \Cat_{1,\uppi} \subset \Catinfty $ preserves finite limits, the induced inclusion $ \Pro(\Cat_{1,\uppi}) \subset \ProCat $ preserves all limits.
\end{remark}

\begin{notation}
	Given a qcqs scheme $ Y $, we write $ \Gal(Y) \in \Pro(\Cat_{1,\uppi}) $ for the \defn{profinite Galois category} of $ Y $ introduced by Barwick--Glasman--Haine \cite[Definitions 10.1.4 \& 12.1.3]{arXiv:1807.03281}.
\end{notation}

\begin{remark}\label{rem:limit_of_Gal_is_Pt}
	Like the étale homotopy type, the profinite Galois category $ \Gal(Y) $ only depends on the étale topos of $ Y $.
	Moreover, the composite
	\begin{equation*}
		\begin{tikzcd}
			\Schqcqs \arrow[r, "\Gal"] & \Pro(\Catinfty) \arrow[r, "\lim"] & \Catinfty
		\end{tikzcd}
	\end{equation*}
	is identified with the functor $ \goesto{Y}{\Pt(Y_{\et})} $.
	With this extra structure of a profinite category, the isomorphism \eqref{intro_eq:endomorphisms_Galois_group} refines to an isomorphism of \textit{profinite} sets.
	See \cite[Lemma 10.3.2 \& Construction 12.1.5]{arXiv:1807.03281}.
\end{remark}

For this article, the details of the definition of $ \Gal(Y) $ are not so important; it is only necessary to know a few of the basic properties of profinite Galois categories.
In the remainder of this subsection, we review all of the properties of profinite Galois categories used in this paper.

\begin{remark}\label{rem:limit_diagram_etale_topoi_and_Gal}
	By \Cref{rem:limit_of_Gal_is_Pt} and \cite[Definition 4.1.5 \& Theorem 10.3.3]{arXiv:1807.03281} the assignment $ \goesto{\Gal(Y)}{\Pt(Y_{\et})} $ is conservative.
	Also note that the functor $ \Pt \colon \fromto{\RTop}{\Cat_{1}} $ preserves limits.
	Therefore, given a diagram \smash{$ Y_{\bullet} \colon \fromto{I^{\smalltriangleleft}}{\Schqcqs} $}, if the induced diagram of étale topoi \smash{$ Y_{\bullet, \et} \colon \fromto{I^{\smalltriangleleft}}{\RTop} $} is a limit diagram, then so is the diagram 
	\begin{equation*}
		\Gal(Y_{\bullet}) \colon \fromto{I^{\smalltriangleleft}}{\ProCat} 
	\end{equation*}
	of profinite Galois categories.
\end{remark}

The following is immediate from the definition of the profinite Galois category.

\begin{observation}\label{obs:Gal_is_profinite_space_for_0-dim}
	Let $ Y $ be a qcqs scheme.
	Then $ \dim(Y) = 0 $ if and only if the profinite category $ \Gal(Y) $ is a profinite $ 1 $-groupoid (i.e., lies in the subcategory $ \ProSpcfin \subset \ProCat $).
\end{observation}

\begin{example}
	Let $ k $ be a field.
	A choice of separable closure $ \kbar \supset k $ provides an equivalence
	\begin{equation*}
		\equivto{\Gal(\Spec(k))}{\BGal(\kbar/k)}
	\end{equation*}
	between the profinite Galois category of $ \Spec(k) $ and the $ 1 $-object profinite $ 1 $-groupoid with profinite automorphism group given by $ \Gal(\kbar/k) $.
\end{example}

A key tool we make use of is the following description of the étale homotopy type in terms of classifying prospaces:

\begin{theorem}[{\cites[Theorem 12.5.1]{arXiv:1807.03281}{arXiv:1812.11637}}]\label{thm:exodromy_definition_of_Piet}
	Let $ Y $ be a qcqs scheme.
	There are natural equivalences of prospaces
	\begin{equation*}
		\equivto{\Pietprotrun(Y)}{\Btrun(\Gal(Y))} \andeq \equivto{\Pietprofin(Y)}{\Bprofin(\Gal(Y))} \period
	\end{equation*}
\end{theorem}

\begin{example}\label{ex:Piet_of_0-dimensional_is_1-truncated}		
	Let $ S $ be a $ 0 $-dimensional qcqs scheme.%
	\footnote{By Serre's cohomological characterization of affineness, every $ 0 $-dimensional qcqs scheme is affine.}
	In light of \Cref{obs:Gal_is_profinite_space_for_0-dim}, \Cref{thm:exodromy_definition_of_Piet} shows that 
	\begin{equation*}
		\Pietprotrun(S) \equivalent \Gal(S) \period
	\end{equation*}
	In particular, the protruncated étale homotopy type $ \Pietprotrun(S) $ is $ 1 $-truncated and profinite.
\end{example}



\section{Galois categories of geometric fibers}\label{sec:Galois_categories_of_fibers}

In this section, we explain why the formation of étale topoi (hence Galois categories, see \Cref{rem:limit_diagram_etale_topoi_and_Gal}) commutes with taking geometric fibers (\Cref{cor:Gal_pullback_geometric_fiber}).
Since the formation of étale topoi does not preserve general pullbacks of schemes \cite[Remark 1.5]{MR4027830}, this is not immediate.
To prove this, we break the problem up into two steps: first we pull back to the strictly henselian local ring, then to the geometric point.

\begin{notation}
	Let $ S $ be a scheme and $ \sbar \to S $ a geometric point.
	We write
	\begin{equation*}
		\Loc{S}{\sbar} \colonequals \Spec(\Ocal_{S,s}^{\sh}) 
	\end{equation*}
	for the \defn{strict localization} of $ S $ at $ \sbar $.
	Given a morphism of schemes $ f \colon X \to S $, we write $ \Xsbar $ and $ \Loc{X}{\sbar} $ for the pullbacks of schemes
	\begin{equation*}
    \begin{tikzcd}[sep=2.25em]
       \Xsbar \arrow[dr, phantom, very near start, "\lrcorner", xshift=-0.25em, yshift=0.12em] \arrow[d] \arrow[r] & \Loc{X}{\sbar}  \arrow[dr, phantom, very near start, "\lrcorner", xshift=-0.25em, yshift=0.25em] \arrow[d] \arrow[r] & X \arrow[d, "f"] \\ 
       \sbar \arrow[r] & \Loc{S}{\sbar} \arrow[r] & S \period
    \end{tikzcd}
  \end{equation*}
\end{notation}

\begin{nul}
	If $ S $ is the spectrum of a field $ k $ and $ \sbar $ is the spectrum of a separable closure $ \kbar \supset k $, then 
	\begin{equation*}
		\Loc{S}{\sbar} = \Spec(\kbar) \andeq \Loc{X}{\sbar} = \Xsbar =  \Spec(\kbar) \crosslimits_{\Spec(k)} X \period
	\end{equation*}
\end{nul}

\begin{proposition}\label{prop:pullback_topos_Loc_geometric_fiber}
	Let $ f \colon \fromto{X}{S} $ be a morphism between qcqs schemes and $ \fromto{\sbar}{S} $ a geometric point.
	Then both of the squares in the diagram of étale topoi 
	\begin{equation*}
      \begin{tikzcd}[sep=2.25em]
	       \Xsbaret \arrow[d] \arrow[r] & \Locet{X}{\sbar} \arrow[d] \arrow[r] & X_{\et} \arrow[d, "\flowerstar"] \\ 
	       \sbar_{\et} \arrow[r] & \Locet{S}{\sbar} \arrow[r] & S_{\et}
      \end{tikzcd}
    \end{equation*}
    are pullback squares in $ \RTop $.
\end{proposition}

\begin{proof}
	First we prove that the right-hand square is a pullback.
	Recall that the strict localization $ \Loc{S}{\sbar} $ is isomorphic (over $ S $) to the limit $ \lim_{U \in \Nbd(\sbar)} U $ over the cofiltered system $ \Nbd(\sbar) $ of \textit{affine} étale neighborhoods of $ \sbar $ in $ S $ \cite[Exposé VIII, 4.5]{MR50:7131}.
	Hence
	\begin{equation*}
		\Loc{X}{\sbar} \isomorphic \lim_{U \in \Nbd(\sbar)} U \cross_{S} X \period
	\end{equation*}
	Since the functor $ (-)_{\et} \colon \Sch \to \RTop $ preserves limits of cofiltered diagrams of qcqs schemes with affine transition morphisms \cites[Exposé VII, Lemme 5.6]{MR50:7131}[Lemma 3.3]{MR4296353} as well as pullbacks along étale morphisms, we see that
	\begin{equation*}
			\Locet{X}{\sbar} \equivalent \lim_{U \in \Nbd(\sbar)} U_{\et} \crosslimits_{S_{\et}} X_{\et} \equivalent S_{(\sbar),\et} \crosslimits_{S_{\et}} X_{\et} \period
	\end{equation*}

	To see that the left-hand square is a pullback, note that the morphism of schemes $ \fromto{\sbar}{\Loc{S}{\sbar}} $ is a closed immersion and the functor $ (-)_{\et} \colon \Sch \to \RTop $ preserves pullbacks along closed immersions.
	See \cites[\HTTthm{Proposition}{7.3.2.12}]{HTT}[\SAGthm{Proposition}{3.1.4.1}]{SAG}[Exposé VIII, Théorème 6.3]{MR50:7131}[Chapter II, Theorem 3.1]{MR559531}.
\end{proof}

Since $ \sbar $ is the spectrum of a separably closed field, we have $ \Gal(\sbar) \equivalent \pt $.
In light of \Cref{rem:limit_diagram_etale_topoi_and_Gal}, \Cref{prop:pullback_topos_Loc_geometric_fiber} implies:

\begin{corollary}\label{cor:Gal_pullback_geometric_fiber}
	Let $ f \colon \fromto{X}{S} $ be a morphism between qcqs schemes and $ \fromto{\sbar}{S} $ a geometric point.
	Then both of the squares in the diagram
	\begin{equation*}
    \begin{tikzcd}[sep=2.25em]
       \Gal(\Xsbar) \arrow[d] \arrow[r] & \Gal(\Loc{X}{\sbar}) \arrow[d] \arrow[r] & \Gal(X) \arrow[d] \\
       \pt \arrow[r] & \Gal(\Loc{S}{\sbar}) \arrow[r] & \Gal(S)
    \end{tikzcd}
  \end{equation*}
  are pullback squares in $ \ProCat $.
\end{corollary}



\section{The fundamental fiber sequence}\label{sec:fiber_sequence}

Let $ f \colon \fromto{X}{S} $ be a morphism between qcqs schemes and $ \fromto{\sbar}{S} $ a geometric point.
We have seen that there is a fiber sequence of profinite categories
\begin{equation*}
	\begin{tikzcd}[sep=1.5em]
		\Gal(\Xsbar) \arrow[r] & \Gal(X) \arrow[r] & \Gal(S) \period
	\end{tikzcd}
\end{equation*}
Our goal is to show that if $ \dim(S) = 0 $, then this fiber sequence remains a fiber sequence after applying the localizations
\begin{equation*}
	\Btrun \colon \fromto{\ProCat}{\ProSpctrun} \andeq \Bprofin \colon \fromto{\ProCat}{\ProSpcfin} \period
\end{equation*}
Since the functors $ \Btrun $ and $ \Bprofin $ do not generally preserve fibers, this is not immediate from the definitions.
Instead, the main technical results of this section are that these localizations preserve pullbacks along morphisms between profinite spaces in the following sense:

\begin{recollection}
	Let $ \Ccal $ be \acategory with pullbacks and $ \Dcal \subset \Ccal $ a full subcategory such that the inclusion admits a left adjoint $ L \colon \fromto{\Ccal}{\Dcal} $.
	We say that the localization $ L $ is \textit{locally cartesian} if for any cospan $ U \to W \ot V $ in $ \Ccal $ with $ U,W \in \Dcal $, the natural map
	\begin{equation*}
		\fromto{L(U \cross_W V)}{U \cross_W L(V)}
	\end{equation*}
	is an equivalence.
	See \cites[\S1.2]{MR3641669}[\S3.2]{MR3570135}.
\end{recollection}

In \cref{subsec:local_cartesianness_classifying_prospaces}, we explain why the classifying prospaces functor $ \Bup \colon \fromto{\ProCat}{\ProSpc} $ is locally cartesian (\Cref{ex:classifying_prospace_locally_cartesian}).
From this and \Cref{cor:Gal_pullback_geometric_fiber}, we deduce a version of the fundamental fiber sequence for classifying prospaces of Galois categories (\Cref{cor:BGal_fundamental_fiber_sequence}).
\Cref{subsec:local_cartesianness_protruncated} shows that the protruncation functor $ \protrun \colon \fromto{\ProSpc}{\ProSpctrun} $ actually preserves all limits (\Cref{prop:protruncation_preserves_limits}).
From this we deduce the fundamental fiber sequence for protruncated étale homotopy types (\Cref{cor:protruncated_fundamental_fiber_sequence}).
In \cref{subsec:local_cartesianness_profinite_completion}, we show that the profinite completion functor is locally cartesian (\Cref{prop:Sigma-completion_locally_cartesian}) and deduce the fundamental fiber sequence for profinite étale homotopy types (\Cref{cor:profinite_fundamental_fiber_sequence}).


\subsection{Local cartesianness of classifying prospaces}\label{subsec:local_cartesianness_classifying_prospaces}

We now show that the localization $ \Bup \colon \fromto{\ProCat}{\ProSpc} $ is locally cartesian.
Using the embedding of $ \Catinfty $ into simplicial spaces, we first treat the localization $ \Bup \colon \fromto{\Catinfty}{\Spc} $.

\begin{nul}
	Let $ \Ical $ be a weakly contractible \category (e.g., $ \Ical = \Deltaop $).
	Let $ \Ccal $ be \acategory with $ \Ical $-shaped colimits.
	Since $ \Ical $ is weakly contractible, the constant functor $ \fromto{\Ccal}{\Fun(\Ical,\Ccal)} $ is fully faithful.
	Hence its left adjoint $ \colim_\Ical \colon \fromto{\Fun(\Ical,\Ccal)}{\Ccal} $ is a localization.
\end{nul}

\noindent The following is a direct reformulation of the definitions.

\begin{lemma}\label{lem:locally_cartesian_via_universal_colimits}
	Let $ \Ical $ be a weakly contractible \category and let $ \Ccal $ be \acategory with $ \Ical $-shaped colimits and pullbacks.
	Then the following are equivalent:
	\begin{enumerate}[label=\stlabel{lem:locally_cartesian_via_universal_colimits}]
		\item $ \Ical $-shaped colimits are universal in the \category $ \Ccal $.

		\item The localization $ \colim_\Ical \colon \fromto{\Fun(\Ical,\Ccal)}{\Ccal} $ is locally cartesian.
	\end{enumerate}
\end{lemma}

\begin{example}\label{ex:classifying_space_locally_cartesain}
	Since geometric realizations of simplicial objects are universal in $ \Spc $, in light of \Cref{rec:complete_Segal_spaces}, the localization $ \Bup \colon \fromto{\Catinfty}{\Spc} $ is locally cartesian.
\end{example}

Now note that passing to \proobjects preserves locally cartesian localizations. 

\begin{lemma}\label{lem:Pro_preserves_local_cartesianness}
	Let $ \Ccal $ be \acategory with pullbacks and let $ L \colon \fromto{\Ccal}{\Dcal} $ be a locally cartesian localization.
	Then the induced localization $ L \colon \fromto{\Pro(\Ccal)}{\Pro(\Dcal)} $ is locally cartesian.
\end{lemma}

\begin{proof}
	We need to show that given a cospan $ U \to W \ot V $ in $ \Pro(\Ccal) $ with $ U, W \in \Pro(\Dcal) $, the natural morphism
	\begin{equation}\label{eq:L_pullback_comparison}
		\fromto{L(U \cross_W V)}{U \cross_W L(V)}
	\end{equation}
	is an equivalence in $ \Pro(\Dcal) $.
	Since $ L \colon \fromto{\Pro(\Ccal)}{\Pro(\Dcal)} $ preserves cofiltered limits, it suffices to prove that \eqref{eq:L_pullback_comparison} is an equivalence in the special case that $ U,W \in \Dcal $ and $ V \in \Ccal $.
	This now follows from the assumption that the localization $ L \colon \fromto{\Ccal}{\Dcal} $ is locally cartesian.
\end{proof}

\begin{example}\label{ex:classifying_prospace_locally_cartesian}
	The localization $ \Bup \colon \fromto{\ProCat}{\ProSpc} $ is locally cartesian.
\end{example}

\begin{corollary}\label{cor:BGal_fundamental_fiber_sequence}
	Let $ f \colon \fromto{X}{S} $ be a morphism between qcqs schemes, and let $ \fromto{\sbar}{S} $ be a geometric point of $ S $.
	If $ \dim(S) = 0 $, then the natural square 
	\begin{equation*}
		\begin{tikzcd}
			\BGal(\Xsbar) \arrow[r] \arrow[d] & \BGal(X) \arrow[d] \\ 
			\pt \arrow[r] & \BGal(S) 
		\end{tikzcd}
	\end{equation*}
	is a pullback square in the \category $ \ProSpc $.
\end{corollary}

\begin{proof}
	Since $ \dim(S) = 0 $, the profinite category $ \Gal(S) $ is a profinite space (\Cref{obs:Gal_is_profinite_space_for_0-dim}).
	The claim follows by applying the locally cartesian localization $ \Bup \colon \fromto{\ProCat}{\ProSpc} $
	to the large pullback square
	appearing in \Cref{cor:Gal_pullback_geometric_fiber}.
\end{proof}


\subsection{Local cartesianness of protruncated classifying spaces}\label{subsec:local_cartesianness_protruncated}

In this subsection, we prove that the protruncation functor preserves all limits and deduce the fundamental fiber sequence for protruncated étale homotopy types (\Cref{cor:protruncated_fundamental_fiber_sequence}).

\begin{proposition}\label{prop:protruncation_preserves_limits}
	The protruncation functor $ \protrun \colon \fromto{\ProSpc}{\ProSpctrun} $ preserves limits.
\end{proposition}

\begin{proof}
	By definition, the functor $ \protrun $ preserves cofiltered limits and the terminal object; hence it suffices to show that $ \protrun $ preserves pullbacks.
	Since $ \protrun $ preserves cofiltered limits, we are reduced to showing that given a cospan $ U \to W \ot V $ of \textit{spaces}, the induced morphism 
	\begin{equation*}\label{eq:protrun_pullback_comparison}
		\fromto{\protrun(U \cross_W V)}{\protrun(U) \crosslimits_{\protrun(W)} \protrun(V)}
	\end{equation*}
	is an equivalence in $ \ProSpctrun $.
	That is, we need to show that for each integer $ k \geq 0 $, the natural morphism
	\begin{equation}\label{eq:trun_leqk_equivalence}
		\trun_{\leq k}(U \cross_W V) \equivalent \trun_{\leq k}\protrun(U \cross_W V)
		\longrightarrow 
		\trun_{\leq k}\bigg(\protrun(U) \crosslimits_{\protrun(W)} \protrun(V) \bigg)
	\end{equation}
	is an equivalence. 
	By definition,
	\begin{equation*}
		\trun_{\leq k}\bigg(\protrun(U) \crosslimits_{\protrun(W)} \protrun(V) \bigg) 
		\equivalent
		\braces{\trun_{\leq k}\bigg(\trun_{\leq n}(U) \crosslimits_{\trun_{\leq n}(W)} \trun_{\leq n}(V) \bigg)}_{n \geq 0} \period
	\end{equation*}
	By \cite[Proposition 4.13]{MR4334846}, the natural map
	\begin{equation*}
		\fromto{U \cross_{W} V}{\trun_{\leq n}(U) \crosslimits_{\trun_{\leq n}(W)} \trun_{\leq n}(V)}
	\end{equation*}
	is $ (n-1) $-connected.
	Hence for $ n \geq k + 1 $, the map  
	\begin{equation*}
		\trun_{\leq k}(U \cross_W V) 
		\longrightarrow 
		\trun_{\leq k}\bigg(\trun_{\leq n}(U) \crosslimits_{\trun_{\leq n}(W)} \trun_{\leq n}(V) \bigg)
	\end{equation*}
	is an equivalence.
	Thus the morphism \eqref{eq:trun_leqk_equivalence} is an equivalence, as desired.
\end{proof}

By \Cref{ex:classifying_prospace_locally_cartesian,prop:protruncation_preserves_limits} we see:

\begin{corollary}\label{cor:protrun_classifying_space_is_localy_cartesian}
	The localization $ \Btrun \colon \fromto{\ProCat}{\ProSpctrun} $ is locally cartesian.
\end{corollary}

\begin{corollary}\label{cor:protruncated_fundamental_fiber_sequence}
	Let $ f \colon \fromto{X}{S} $ be a morphism between qcqs schemes, and let $ \fromto{\sbar}{S} $ be a geometric point of $ S $.
	If $ \dim(S) = 0 $, then the naturally null sequence 
	\begin{equation*}
		\begin{tikzcd}[sep=1.5em]
			\Pietprotrun(\Xsbar) \arrow[r] & \Pietprotrun(X) \arrow[r] & \Pietprotrun(S) 
		\end{tikzcd}
	\end{equation*}
	is a fiber sequence in the \category $ \ProSpctrun $.
\end{corollary}

\begin{proof}
	Combine \Cref{thm:exodromy_definition_of_Piet,cor:BGal_fundamental_fiber_sequence,prop:protruncation_preserves_limits}.
\end{proof}


\subsection{Local cartesianness of profinite completion}\label{subsec:local_cartesianness_profinite_completion}

We now explain why profinite completion is locally cartesian (\Cref{prop:Sigma-completion_locally_cartesian}).
From this we deduce the fundamental fiber sequence for profinite étale homotopy types (\Cref{cor:profinite_fundamental_fiber_sequence}).

Since the proof is exactly the same (and we need it in future work), we record the more general statement that completion at a set of primes is locally cartesian. 
To do so, we first introduce some definitions.

\begin{definition}\label{def:Sigma-finite}
	Let $ \Sigma $ be a set of prime numbers.
	\begin{enumerate}[label=\stlabel{def:Sigma-finite}, ref=\arabic*]
		\item A finite group $ G $ is an \defn{\Sigmagroup} if the order of $ G $ is in the multiplicative closure of $ \Sigma $. 

		\item A space $ U $ is \defn{\Sigmafinite} if $ U $ is \pifinite and all homotopy groups of $ U $ are \Sigmagroups.
		We write $ \SpcSigma \subset \Spcfin $ for the full subcategory spanned by the \Sigmafinite spaces.
	\end{enumerate} 
\end{definition}

\begin{notation}[{(\Sigmacompletion)}]
	The inclusion $ \ProSpcSigma \subset \ProSpc $ admits a left adjoint
	\begin{equation*}
		(-)\Sigmacomp \colon \fromto{\ProSpc}{\ProSpcSigma}
	\end{equation*}
	called \defn{\Sigmacompletion}.
	Write \smash{$ \BSigmacomp $} for the composite
	\begin{equation*}
		\begin{tikzcd}[sep=2em]
			\ProCat \arrow[r, "\Bup"] & \ProSpc \arrow[r, "(-)\Sigmacomp"] & \ProSpcSigma \period
		\end{tikzcd}
	\end{equation*}
\end{notation}

\begin{observation}\label{obs:profinite_completion_of_spaces_left_adjoint}
	In light of \Cref{obs:existence_of_materalization}, the composite \Sigmacompletion functor
	\begin{equation*}
		\begin{tikzcd}[sep=2em]
			\Spc \arrow[r, hooked] & \ProSpc \arrow[r, "(-)\Sigmacomp"] & \ProSpcSigma 
		\end{tikzcd}
	\end{equation*}
	is a left adjoint.
\end{observation}

In order to show that \Sigmacompletion is locally cartesian, we make use of the following generalization of \cite[\SAGthm{Theorem}{E.6.0.7} \& \SAGthm{Corollary}{E.6.0.8}]{SAG}.
See also \cite[\href{https://www.math.ias.edu/~lurie/papers/DAG-XIII.pdf\#page=74}{Proposition 3.2.4}]{DAGXIII}.

\begin{theorem}\label{thm:pro-Sigma-finite_straightening_unstraightening}
	Let $ \Sigma $ be a set of prime numbers.
	\begin{enumerate}[label=\stlabel{thm:pro-Sigma-finite_straightening_unstraightening}, ref=\arabic*]
		\item Let $ U $ be an \Sigmafinite space.
		Then the functor
		\begin{equation*}
			\textstyle \colim_U \colon \fromto{\Fun(U,\ProSpcSigma)}{\ProSpcSigma_{/U}}
		\end{equation*}
		is an equivalence of \categories.

		\item Given a map $ \fromto{U}{W} $ of \Sigmafinite spaces, the functor
		\begin{equation*}
			U \cross_W (-) \colon \fromto{\ProSpcSigma_{/W}}{\ProSpcSigma_{/U}}
		\end{equation*}
		preserves limits and colimits.
	\end{enumerate}
\end{theorem}

\noindent Lurie only states \Cref{thm:pro-Sigma-finite_straightening_unstraightening} when $ \Sigma $ is the set of all primes (so $ \SpcSigma = \Spcfin $) or a single prime.
However, the proofs given in \cite[\SAGsubsec{E.6.1} \& \SAGsubsec{E.6.2}]{SAG} work \textit{verbatim} in this more general setting.

The next lemma helps us compare \Sigmacompletions of pullbacks with pullbacks of \Sigmacompletions.

\begin{lemma}\label{lem:pullback_along_Sigma-fintie}
	Let $ \Sigma $ be a set of prime numbers and $ \fromto{U}{W} $ a map of spaces.
	Then:
	\begin{enumerate}[label=\stlabel{lem:pullback_along_Sigma-fintie}, ref=\arabic*]
		\item\label{lem:pullback_along_Sigma-fintie.1} The functor $ (U \cross_{W} (-))\Sigmacomp \colon \fromto{\Spc}{\ProSpcSigma} $ preserves colimits.

		\item\label{lem:pullback_along_Sigma-fintie.2} If $ U $ and $ W $ are \Sigmafinite, then the functor $ U \cross_{W} (-)\Sigmacomp \colon \fromto{\Spc}{\ProSpcSigma} $ preserves colimits.
	\end{enumerate}
\end{lemma}

\begin{proof}
	For \enumref{lem:pullback_along_Sigma-fintie}{1}, note that colimits are universal in $ \Spc $ and the functor $ (-)\Sigmacomp \colon \fromto{\Spc}{\ProSpcSigma} $ preserves colimits (\Cref{obs:profinite_completion_of_spaces_left_adjoint}).

	For \enumref{lem:pullback_along_Sigma-fintie}{2}, note that since $ U $ and $ W $ are \Sigmafinite, the pullback functor 
	\begin{equation*}
		U \cross_{W} (-) \colon \fromto{\ProSpcSigma_{/W}}{\ProSpcSigma_{/U}}
	\end{equation*}
	preserves colimits.
	Thus the claim follows from the fact that the functor $ (-)\Sigmacomp \colon \fromto{\Spc}{\ProSpcSigma} $ preserves colimits.
\end{proof}

\begin{proposition}\label{prop:Sigma-completion_locally_cartesian}
	Let $ \Sigma $ be a set of prime numbers.
	Then the localization
	\begin{equation*}
		(-)\Sigmacomp \colon \fromto{\ProSpc}{\ProSpcSigma}
	\end{equation*}
	is locally cartesian.
\end{proposition}

\begin{proof}
	Given a cospan $ U \to W \ot V $ in $ \ProSpc $ with $ U,W \in \ProSpcSigma $, we need to show that the natural map
	\begin{equation}\label{eq:Sigma-completion_comparison_pullbacks}
		\fromto{(U \cross_W V)\Sigmacomp}{U \cross_W V\Sigmacomp} 
	\end{equation}
	is an equivalence.
	Since \Sigmacompletion preserves cofiltered limits, we are reduced to the case where $ U,W \in \SpcSigma $ and $ V \in \Spc $.
	In this case, \Cref{lem:pullback_along_Sigma-fintie} shows that both sides of \eqref{eq:Sigma-completion_comparison_pullbacks} preserve colimits in $ V $.
	Since $ \Spc $ is generated under colimits by the point, we are reduced to showing that \eqref{eq:Sigma-completion_comparison_pullbacks} is an equivalence when $ V = \pt $; this is true because $ \pt $ is \Sigmafinite.
\end{proof}

By \Cref{ex:classifying_prospace_locally_cartesian,prop:Sigma-completion_locally_cartesian} we see:

\begin{corollary}\label{cor:Sigma-complete_classifying_space_is_localy_cartesian}
	Let $ \Sigma $ be a set of prime numbers.
	Then the localization $ \BSigmacomp \colon \fromto{\ProCat}{\ProSpcSigma} $ is locally cartesian.
\end{corollary}

\begin{corollary}\label{cor:profinite_fundamental_fiber_sequence}
	Let $ f \colon \fromto{X}{S} $ be a morphism between qcqs schemes, and let $ \fromto{\sbar}{S} $ be a geometric point of $ S $.
	If $ \dim(S) = 0 $, then the naturally null sequence 
	\begin{equation}\label{eq:profinite_fundamental_fiber_sequence}
		\begin{tikzcd}[sep=1.5em]
			\Pietprofin(\Xsbar) \arrow[r] & \Pietprofin(X) \arrow[r] & \Pietprofin(S) 
		\end{tikzcd}
	\end{equation}
	is a fiber sequence in the \category $ \ProSpcfin $.
\end{corollary}

\begin{proof}
	Combine \Cref{thm:exodromy_definition_of_Piet,cor:BGal_fundamental_fiber_sequence,prop:Sigma-completion_locally_cartesian}.
\end{proof}

\begin{warning}\label{warning:Sigma-completion_does_not_preserve_fundamental_fiber_sequence}
	The fiber sequence \eqref{eq:profinite_fundamental_fiber_sequence} need not remain a fiber sequence after completion at a set of primes.
	To see this, let $ k $ be a field with separable closure $ \kbar \supset k $ and absolute Galois group $ \Gup \colonequals \Gal(\kbar/k) $.
	Set $ S \colonequals \Spec(k) $ and $ X \colonequals \Spec(\kbar) $.
	Note that since $ \Pietprofin(\Xkbar) \equivalent \Omega\BG $ is a profinite set, it is already \Sigmacomplete.
	Write $ \Gup^{\Sigma} $ for the maximal \proSigma quotient of $ \Gup $.
	In this case, \cite[Corollary 3.7]{MR0245577} implies that the natural map
	\begin{equation*}
		\Omega\BG \equivalent \Piet(\Xkbar)\Sigmacomp \longrightarrow \Omega\paren{\Piet(S)\Sigmacomp} \equivalent \Omega\paren{(\BG)\Sigmacomp}
	\end{equation*}
	induces the quotient map $ \surjto{\Gup}{\Gup^{\Sigma}} $ on $ \uppi_0 $.
\end{warning}

Using the local cartesianness of \Sigmacompletion, we see that the failure of $ \Gal(\kbar/k) $ to be a \proSigma group is the only obstruction to \eqref{eq:profinite_fundamental_fiber_sequence} remaining a fiber sequence after \Sigmacompletion: 

\begin{definition}
	Let $ \Sigma $ be a set of prime numbers and let $ k $ be a field.
	We say that $ k $ is \defn{\Sigmaclosed}%
	\footnote{For a prime $ p $, the notion of a $ p $-closed field used here is stronger than the one introduced in \cite[Chapter VI, \S1]{MR2392026}.}
	if for every finite Galois extension $ K \supset k $ and prime $ \el \in \Sigma $, the degree of $ K $ over $ k $ is \textit{not} divisible by $ \el $.
\end{definition}

\begin{nul}
	Given a set of prime numbers $ \Sigma $, write $ \Sigma' $ for the complement of $ \Sigma $ in the set of all primes.
	By the fundamental theorem of Galois theory, a field $ k $ is \Sigmaprimeclosed if and only if for any separable closure $ \kbar \supset k $, the Galois group $ \Gal(\kbar/k) $ is a \proSigma group.
\end{nul}

Logic provides a source of examples of \Sigmaclosed fields:

\begin{example}\label{ex:Sigma-closed_fields}
	\hfill
	\begin{enumerate}[label=\stlabel{ex:Sigma-closed_fields}, ref=\arabic*]
		\item If $ k $ is a real closed field, then $ k $ is $ 2' $-closed.

		\item Let $ k $ be a field of characteristic $ p > 0 $. 
		If $ k $ is infinite and does not have the independence property (i.e., is a \textit{NIP field}), then $ k $ is $ p $-closed \cites[Corollary 4.4]{MR2837131}{Scanlon:stable_fields_are_Artin-Schreier_closed}
	\end{enumerate}
\end{example}

\begin{observation}\label{obs:Sigma-completeness_for_0-dimensional_schemes}
	Let $ S $ be a $ 0 $-dimensional qcqs scheme and let $ \Sigma $ be a set of prime numbers.
	Since the profinite étale homotopy type $ \Pietprofin(S) $ is a profinite $ 1 $-groupoid with automorphism groups the absolute Galois groups of the residue fields of $ S $, the profinite space $ \Pietprofin(S) $ is \Sigmacomplete if and only if each residue field of $ S $ is \Sigmaprimeclosed.
\end{observation}

\begin{corollary}
	Let $ f \colon \fromto{X}{S} $ be a morphism between qcqs schemes, let $ \fromto{\sbar}{S} $ be a geometric point of $ S $, and let $ \Sigma $ be a set of prime numbers.
	If $ \dim(S) = 0 $ and each residue field of $ S $ is \Sigmaprimeclosed, then the naturally null sequence 
	\begin{equation*}
		\begin{tikzcd}[sep=1.5em]
			\Piet(\Xsbar)\Sigmacomp \arrow[r] & \Piet(X)\Sigmacomp \arrow[r] & \Piet(S)\Sigmacomp 
		\end{tikzcd}
	\end{equation*}
	is a fiber sequence in the \category $ \ProSpcSigma $.
\end{corollary}

\begin{proof}
	By assumption, $ \Pietprofin(S) $ is \Sigmacomplete; thus the conclusion follows from \Cref{prop:Sigma-completion_locally_cartesian,cor:profinite_fundamental_fiber_sequence}.
\end{proof}


\DeclareFieldFormat{labelnumberwidth}{#1}
\printbibliography[keyword=alph, heading=references]
\DeclareFieldFormat{labelnumberwidth}{{#1\adddot\midsentence}}
\printbibliography[heading=none, notkeyword=alph]

\end{document}